\documentclass[journal]{IEEEtran}

\usepackage{color}
\usepackage{amsmath,amssymb,amsfonts,amsthm}
\usepackage{url,hyperref}
\usepackage[all]{xy}
\usepackage{algorithm}
\usepackage{algpseudocode}
\usepackage{graphicx}

\usepackage{adjustbox}
\usepackage{fancyvrb}

\newcommand{\bbR}{{\mathbb R}}

\IEEEeqnarraydefcolsep{0}{\leftmargini}
\usepackage{caption}

\newtheorem{thm}{Theorem}%[section]

\newtheorem{lem}{Lemma}
\newtheorem{cor}{Corollary}
\newtheorem{claim}{Claim}
\newtheorem*{claim*}{Claim}

\newtheorem*{note*}{Notation}
\newtheorem{gs}{Assumption}

\newtheorem{ex}{Counter-example}
\newtheorem{rem}{Remark}

\def \bfE{{\bf E}}

\def \bbR {\mathbb R}

\begin{document}

%\setlength{\baselineskip}{17pt}
%%%%%%%%%%%%%%%%%%%%%%%%%%%%%%%%%%%%%%%%%%%%%%%%%%%%%%%%%%%%%

\title{Communication-Efficient Distributed SGD with Error-Feedback, Revisited$^*$}

\author{Tran Thi Phuong, Le Trieu Phong% <-this % stops a space
\IEEEcompsocitemizethanks{ \IEEEcompsocthanksitem $^*$This paper is accepted to International Journal of Computational Intelligence Systems, Volume 14, Issue 1, 2021, Pages 1373 -- 1387. DOI: \url{https://doi.org/10.2991/ijcis.d.210412.001}. \IEEEcompsocthanksitem T. T. Phuong is with the Faculty of Mathematics and Statistics, Ton Duc Thang University, Ho Chi Minh City, Vietnam. Postal address: 19 Nguyen Huu Tho street, Tan Phong ward, District 7, Ho Chi Minh City, Vietnam. She is also with Meiji University (Japan) and NICT (Japan). Email: {\tt tranthiphuong@tdtu.edu.vn}. \IEEEcompsocthanksitem L. T. Phong is with the National Institute of Information and Communications Technology (NICT). Postal address:  4-2-1, Nukui-Kitamachi, Koganei, Tokyo 184-8795, Japan.  Email: {\tt phong@nict.go.jp}.}% <-this % stops an unwanted space
\thanks{}
}

\maketitle

%
%
%\author{Tran Thi Phuong$^{(*,**,***)}$}
%\address[$*$]{Faculty of Mathematics and Statistics, Ton Duc Thang University, Ho Chi Minh City, Vietnam. Postal address: 19 Nguyen Huu Tho street, Tan Phong ward, District 7, Ho Chi Minh City, Vietnam.}
%\address[$**$]{Meiji University. Postal address: 1-1-1 Higashi-Mita, Tama-ku, Kawasaki-shi, Kanagawa 214-8571, Japan.}
%\email{tranthiphuong@tdtu.edu.vn}
%
%
%\author{Le Trieu Phong$^{(***)}$}
%\address[$***$]{National Institute of Information and Communications Technology (NICT). Postal address: 4-2-1, Nukui-Kitamachi, Koganei, Tokyo 184-8795, Japan.}
%\email{phong@nict.go.jp}
%\thanks{$^\star$A version of this paper appears at IEEE Access DOI: \href{https://ieeexplore.ieee.org/document/8713445?source=authoralert}{10.1109/ACCESS.2019.2916341}}
%

%\IEEEtitleabstractindextext{%
\begin{abstract}
We show that the convergence proof of a recent algorithm called dist-EF-SGD for distributed stochastic gradient descent with communication efficiency using error-feedback of Zheng et al. (NeurIPS 2019) is problematic mathematically. Concretely, the original error bound for arbitrary sequences of learning rate is unfortunately incorrect, leading to an invalidated upper bound in the convergence theorem for the algorithm. As evidences, we explicitly provide several counter-examples, for both convex and non-convex cases, to show the incorrectness of the error bound. We  fix the issue by providing a new error bound and its corresponding proof, leading to  a new convergence theorem for the dist-EF-SGD algorithm, and therefore recovering its mathematical analysis.
\end{abstract}

\begin{IEEEkeywords}
Optimizer, distributed learning, SGD, error-feedback, deep neural networks.
\end{IEEEkeywords}
%}

\section{Introduction}

\subsection{Background}
For training deep neural networks over large-scale and distributed datasets, distributed stochastic gradient descent (distributed SGD) is a vital method. In distributed SGD, a central server updates the model parameters using information transmitted from distributed workers, as illustrated in Figure \ref{network_fault_intro}.

\begin{figure*}
\centering
\includegraphics[scale=0.7]{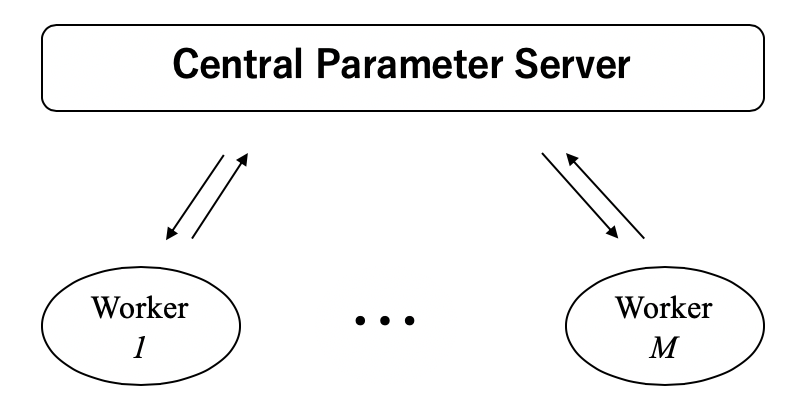}
\caption{The computation model of distributed SGD. {\color{black} Multiple workers communicate with a central parameter server synchronously. Each worker, after local computations on its data, uploads selected results to the server. The server aggregates all uploaded results from the workers, and sends back the aggregated result from its computations to all workers. These are iterated for multiple rounds.}}\label{network_fault_intro}
\end{figure*}

Communication between the server and distributed workers can be a bottleneck in distributed SGD. 
Alleviating the bottleneck is a considerable concern of the community, so that variants of distributed SGD using gradient compression have been proposed  to reduce the communication cost between workers and the server.

Recently, Zheng et al. \cite{Zhengetal} proposed an algorithm named dist-EF-SGD recalled in Algorithm \ref{dist-EF-SGD}, in which gradients are compressed before transmission, and errors between real and compressed gradients in one step of the algorithm are re-used in future steps. 

\subsection{Our contributions} 
In this paper, we point out a flaw in the convergence proof of Algorithm \ref{dist-EF-SGD} given in Zheng et al. \cite{Zhengetal}. We then fix the flaw by providing a new convergence theorem with a new proof for Algorithm \ref{dist-EF-SGD}.

\begin{algorithm*}[t]
    \caption{Distributed SGD with Error-Feedback (dist-EF-SGD) \cite{{Zhengetal}}}\label{dist-EF-SGD}
    \begin{algorithmic}[1]
    	\State \textbf{Input:} Loss function $\mathcal L$,  learning rate  $\{\eta_t\}$ with $\eta_{-1} = 0$; number of workers $M$; compressor $\mathcal C(\cdot)$
	\State \textbf{Initialize:} initial parameter $x_0\in \mathbb R^d$; error $e_{0,i} = 0 \in \mathbb R^d$ on each worker $i$; error $\tilde{e}_0 = 0\in \mathbb R^d$ on server
	\For {$t\in \{0, \dots, T-1\}$}
		\State {$\bullet$ \bf on each worker $1\le i \le M$:} 
			\State \qquad pick data $ \xi_{t,i}$  from the dataset
			\State \qquad$g_{t,i} = \nabla {\mathcal L}(x_t, \xi_{t,i})$   \Comment{stochastic gradient}
			\State \qquad $p_{t,i} =  g_{t,i} + \frac{\eta_{t-1}}{\eta_t}e_{t,i}$ \label{p_t_i_}\Comment{gradient added with previous error}
			\State \qquad {push $\Delta_{t,i} = \mathcal C(p_{t,i})$ to server} \Comment{gradient compression at worker, and transmission}
			\State  \qquad {pull $\tilde{\Delta}_t$ from server}
			\State  \qquad $x_{t+1} = x_t - \eta_t \tilde{\Delta}_t$  \Comment{local weight update}
			\State  \qquad $e_{t+1,i} = p_{t,i} - \Delta_{t,i} $ \Comment{local error-feedback to next step}
		\State {$\bullet$ \bf on central parameter server:}
		\State  \qquad {pull $\Delta_{t,i}$ from each worker $i$}
		\State  \qquad $\tilde p_t = \frac{1}{M}\sum_{i=1}^M\Delta_{t,i} + \frac{\eta_{t-1}}{\eta_t}\tilde e_t$ \Comment{gradient average with error}
		\State  \qquad {push $\tilde{\Delta}_t = \mathcal C(\tilde p_t)$ to each worker} \Comment{gradient compression at server}
		\State  \qquad $\tilde e_{t+1} = \tilde p_t - \tilde{\Delta}_t $ \Comment{error on server}
	\EndFor

    \end{algorithmic}
\end{algorithm*}

Zheng et al. \cite{{Zhengetal}} stated the following theorem for any sequence of learning rate $\{\eta_t\}$. %, i.e., the sequence $\{\eta_t\}$ can be either increasing or increasing. 

\medskip
\noindent
{\bf Theorem A {(Theorem 1 of {\cite{Zhengetal}, {\bf problematic}}}).} 
{\it
Suppose that Assumptions 1-3 (given together with related notations in Section \ref{prel}) hold. Assume that the learning rate $0<\eta_t<\frac{3}{2L}$ for all $t\ge 0$. For sequence $x_t$ generated from Algorithm \ref{dist-EF-SGD}, we have the following upper bound on the expected Euclidean norm of gradients
\begin{align*}
\lefteqn{{\mathbf E}[\lVert \nabla f(x_o)\rVert^2] \le \frac{4(f(x_0) - f^{\star})}{\sum_{k=0}^{T-1}\eta_k(3-2L\eta_k)} }\\
&+\frac{2L\sigma^2}{M}\sum_{t=0}^{T-1}\frac{\eta_t^2}{\sum_{k=0}^{T-1}\eta_k(3-2L\eta_k)}\\
&+\frac{32L^2(1-\delta)G^2}{\delta^2}\left[ 1 + \frac{16}{\delta^2}\right]\sum_{t=0}^{T-1}\frac{\eta_t\eta_{t-1}^2}{\sum_{k=0}^{T-1}\eta_k(3-2L\eta_k)},
\end{align*} 
where $o\in \{0,...,T-1\}$ is an index such that the probability 
$$\Pr(o=k) = \frac{\eta_k(3-2L\eta_k)}{\sum_{t=0}^{T-1}\eta_t(3-2L\eta_t)}, \forall k = 0,...,T-1.$$
}\\
{ \noindent \bf Problem in Theorem A.} Unfortunately, the proof of Theorem A as given in  \cite{Zhengetal} becomes  invalidated when the learning rate sequence $\{\eta_t\}$ is decreasing. In that proof, a lemma is employed to handle decreasing learning rate sequences. However, in Section \ref{counterexample} we present several counter-examples showing that  lemma does not hold. We move on to fix that lemma and finally obtain the following result as our correction for Theorem A. 

\begin{thm} (Our correction  for Theorem A) \label{mainthmcase1} With all notations and assumptions are identical to Theorem A, we have
\begin{align*}
\lefteqn{{\mathbf E}[\lVert \nabla f(x_o)\rVert^2] \le \frac{4(f(x_0) - f^{\star})}{\sum_{k=0}^{T-1}\eta_k(3-2L\eta_k)} }\\
&+\frac{2L\sigma^2}{M}\sum_{t=0}^{T-1}\frac{\eta_t^2}{\sum_{k=0}^{T-1}\eta_k(3-2L\eta_k)}\\
&+\frac{8(1-\delta)(2-\delta)G^2L^2}{\delta\sum_{k=0}^{T-1}\eta_k(3-2L\eta_k)}\sum_{t=0}^{T-1}\eta_t\eta_{t-1}^2\sum_{k=0}^{t-1}\frac{\eta_{t-1-k}^2}{\eta_{t-1}^2}\alpha^k\\
&+\frac{16(1-\delta)(2-\delta)^3G^2L^2}{\delta^2\sum_{k=0}^{T-1}\eta_k(3-2L\eta_k)}\times\\
&\quad\quad\quad\quad\quad\sum_{t=0}^{T-1}\eta_t\eta_{t-1}^2\sum_{j=0}^{t-1} \alpha^{t-1-j}\sum_{k=0}^{j}\frac{\eta_{j-k}^2}{\eta_{t-1}^2}\alpha^k,
\end{align*}
where $\alpha = 1-\frac{\delta}{2}$ and $o\in \{0,...,T-1\}$ is an index such that the probability 
$$\Pr(o=k) = \frac{\eta_k(3-2L\eta_k)}{\sum_{t=0}^{T-1}\eta_t(3-2L\eta_t)}, \forall k = 0,...,T-1.$$ 
\end{thm}
\medskip\noindent
In addition, we show that the upper bound in Theorem \ref{mainthmcase1} becomes $O\left(\frac{1}{\sqrt{MT}}\right)$ for a proper choice of decreasing sequence $\{\eta_t\}$ in Corollary \ref{decreasing}. Moreover the upper bound in Theorem \ref{mainthmcase1} matches previous results given in Zheng et al. \cite{Zhengetal} when $\{\eta_t\}$ is non-decreasing (Corollary \ref{errorcor}).

\subsection{Paper roadmap} We begin with notations and settings in Section \ref{prel}. In Section \ref{counterexample}, we provide counterexamples to justify the issue in \cite{Zhengetal} for both non-convex and convex cases. We then correct the issue in Section \ref{correctingerror} and then present a proof for Theorem \ref{mainthmcase1} in Section \ref{corectmainthm}.

{\color{black}\subsection{Related works} %literature review
The use of gradient compression for reducing the communication cost is widely considered in distributed machine learning recently. One line of research is to compress the gradient only on the worker side before sending the result to the parameter server, namely one-side compression. The parameter server receives and aggregates these results and sends back the aggregated result to all workers. Some recent papers such as \cite{BernsteinZAA19}, \cite{DebrajBasu_nips2019}, \cite{NIPS2019_9571}, \cite{StichCJ18} are in this line of research. 

Another line of research uses gradient compression on both workers and server, namely two-side compression. In these two-side compression methods, the workers send the compressed local gradients or some corrected forms of them to the parameter server, and the parameter server compresses the aggregated result before sending it back to all workers. Papers \cite{Zhengetal}, \cite{tang19d_icml}, and \cite{pmlr-v108-liu20a} use two-side compression with an identical method of gradient compression for both workers and the parameter server. Paper \cite{9051706} considers two-side compression with flexible compression for both workers and the parameter server.
}
\section{Preliminaries}\label{prel}
Let $\langle \cdot,\cdot \rangle$ be the inner product of vectors. The Cauchy-Schwarz inequality states that for all vectors $u,v$ it holds that $|\langle u,v\rangle|^2 \le \langle u,u \rangle \times \langle v,v\rangle$. The Young inequality with $\gamma>0$ (sometimes called the Peter-Paul inequality) states that $(a+b)^2 \le (1+\gamma)a^2 + (1+1/\gamma)b^2\ \forall a,b\in \bbR$. Let $\lVert\cdot \rVert$ be the Euclidean norm of a  vector. 

 {\color{black} For completeness, we recall the algorithm of Zheng et al. \cite{Zhengetal} in  Algorithm \ref{dist-EF-SGD} and its explanation as follows.  At iteration $t$, the scale $ \frac{\eta_{t-1}}{\eta_t}$ of the local accumulated error vector $e_{t,i}$ is added to the gradient $g_{t,i}$ (line 7 of Algorithm \ref{dist-EF-SGD}) for the compression step. Each worker $i$ stores these local accumulated error vector $e_{t,i}$ and local corrected gradient vector $p_{t,i}$ for the next iteration. The compressed $\Delta_{t,i}$ of $p_{t,i}$ are pushed to the parameter server. The parameter server aggregates these $\Delta_{t,i}$ and uses the aggregated result to update the global error-corrected vector $\tilde p_t $  {\color{black} (line 14 of Algorithm \ref{dist-EF-SGD})}, which in turn is used to update the global accumulated error vector $\tilde e_{t+1}$ {\color{black}(line 16 of Algorithm \ref{dist-EF-SGD})}. Each worker receives the compressed $\tilde{\Delta}_t$ of $\tilde p_t $ from the parameter server and uses it to update the parameter $x_{t+1}$.
 
 In order to construct Algorithm \ref{dist-EF-SGD}, Zheng et al. \cite{Zhengetal} used the idea of Karimireddy et al. \cite{EFSGD} that combined gradient compression with error correction. The innovative ideas of  Zheng et al. \cite{Zhengetal}  were to apply compression on the parameter server and to use the scale $ \frac{\eta_{t-1}}{\eta_t}$ in line 7 of Algorithm \ref{dist-EF-SGD}.  Unfortunately the scale $ \frac{\eta_{t-1}}{\eta_t}$ caused an issue in the proof of convergence theorem of Algorithm \ref{dist-EF-SGD}. We examine this issue in details  in Section \ref{counterexample}. }

\subsection{Compressor and Assumptions}
Following \cite{StichCJ18,EFSGD}, an operator $\mathcal C: \mathbb R^d \to \mathbb R^d$ is a $\delta$-compressor for a number $\delta\in(0,1)$ if 
\begin{eqnarray}\label{comp}
{\mathbf E}_{\mathcal C}\lVert\mathcal C(x) -x\rVert^2 \le (1-\delta)\lVert x\rVert^2
\end{eqnarray}
where the expectation ${\mathbf E}_{\mathcal C}$ is taken over the randomness of ${\mathcal C}$.

Given a loss function ${\mathcal L}$, define $f(x) = {{\mathbf E}}_{\xi}[ {\mathcal L}(x,\xi)]$ where $x \in {\mathbb  R}^d$ is the (neural network) model parameters, and $\xi$ is the data batch drawn from some  unknown distribution.  We consider the following assumptions on $f$, which are standard and have been used in previous works \cite{Zhengetal,EFSGD}. %In order to examine the convergence of \flexComp, we consider the following assumptions.

\begin{gs} \label{smoothass} $f$ is lower-bounded, i.e., $f^{\star} = \inf_{x\in \mathbb R^d} f(x)<\infty$, and $L$-smooth i.e., $f$ is differentiable and there exists a constant $L\ge 0$ such that 
\begin{eqnarray}\label{Lip}
\lVert \nabla f(x) - \nabla f(y) \rVert \le L \lVert x-y \rVert, \quad \forall x,y\in \mathbb R^d.
\end{eqnarray}
\end{gs}

By \cite{Nesterov}, the $L$-smooth condition in (\ref{Lip}) implies that $\forall x,y\in \mathbb R^d$,
\begin{eqnarray}\label{Lippro}
f(x) \le f(y) + \langle \nabla f(y), x-y\rangle + \frac{L}{2}\lVert x-y\rVert^2.
\end{eqnarray}

\begin{gs} 
\label{stobounded}
Let ${\mathbf E}_t$ denote the expectation at iteration $t$. 
Then ${\mathbf E}_t[g_{t,i}] = \nabla f (x_t) $ and the stochastic gradient $g_{t,i}$ has bounded gradient, i.e.,  $${\mathbf E}_t[\lVert g_{t,i} -\nabla f(x_t)\rVert^2] \le \sigma^2.$$
\end{gs}
\begin{gs} 
\label{fullgradbounded}
The full gradient $\nabla f$ is uniformly bounded, i.e., $\lVert \nabla f(x_t)\rVert^2 \le \omega^2$.
\end{gs}

\medskip\noindent
Under Assumptions \ref{stobounded} and \ref{fullgradbounded}, we have 
\begin{eqnarray}\label{for_gti}
{\mathbf E}_t[\lVert g_{t,i} \rVert^2] \le G^2 =\sigma^2 + \omega^2,
\end{eqnarray} 
because ${\mathbf E}_t[\lVert g_{t,i} -\nabla f(x_t)\rVert^2] \le \sigma^2$, $\lVert \nabla f(x_t)\rVert^2 \le \omega^2$, and  the fact that $\bfE[\lVert X-\bfE[X]\rVert^2] + \lVert \bfE[X]\rVert^2= \bfE[\lVert X\rVert^2]$.

\subsection{Supporting lemmas}
We need a few supporting lemmas for proving Theorem \ref{mainthmcase1}.
%%%%%%%%%%%%%
\begin{lem}\label{normineq}
Let $0<M\in \mathbb N$ and $x_i \in \mathbb R^d$. Then 
$$\left\lVert \frac{1}{M}\sum_{i=1}^{M}x_i\right\rVert^2 \le  \frac{1}{M}\sum_{i=1}^{M}\left\lVert x_i\right\rVert^2.$$
\end{lem}
\begin{proof}~
Since $x_i \in \mathbb R^d$, $x_i$ has the form $x_i = (x_{i,1}, x_{i,2},..., x_{i,d})\in \mathbb R^d$. We have
\begin{align*}  
\left\lVert \frac{1}{M}\sum_{i=1}^{M}x_i\right\rVert^2  
&= \frac{1}{M^2}\left\lVert \sum_{i=1}^{M}x_i\right\rVert^2\\
&=  \frac{1}{M^2}\left\lVert \left(\sum_{i=1}^{M}x_{i, 1}, \sum_{i=1}^{M}x_{i,2},..., \sum_{i=1}^{M}x_{i,d}\right)\right\rVert^2 \\
&=  \frac{1}{M^2} \sum_{j=1}^d \left(\sum_{i=1}^{M}x_{i,j}\right)^2.
\end{align*} 
Applying the Cauchy-Schwarz inequality on $\left(\sum_{i=1}^{M}x_{i,j}\right)^2$ gives us
$$\left(\sum_{i=1}^{M}x_{i,j}\right)^2 \le M \sum_{i=1}^{M}x_{i,j}^2.$$
Therefore 
\begin{align*}  
\left\lVert \frac{1}{M}\sum_{i=1}^{M}x_i\right\rVert^2 
& \le  \frac{1}{M}\sum_{j=1}^d \left(\sum_{i=1}^{M}x_{i,j}^2\right)\\
&= \frac{1}{M}\sum_{i=1}^{M} \left(\sum_{j=1}^d x_{i,j}^2\right)\\
&= \frac{1}{M}\sum_{i=1}^{M} \lVert x_i \rVert^2
\end{align*} 
which ends the proof.
\end{proof}

\begin{lem} \label{recseq} Let $\{a_t\}, \{\alpha_t\}, \{\beta_t\}$ {\color{black}be} non-negative sequences in $\mathbb R$ such that $a_0 = 0$ and, for all $t\ge 0$, 
\begin{eqnarray} \label{at}
a_{t+1} \le \alpha_t a_t + \beta_t.
\end{eqnarray} 
Then $$a_{t+1} \le \beta_t + \sum_{j=1}^t \prod_{i=j}^t\alpha_i\beta_{j-1}.$$
In particular, if $\beta_t = \beta$ for all $t$, then 
$$a_{t+1} \le \beta\left(1 + \sum_{j=1}^t \prod_{i=j}^t\alpha_i\right).$$
\end{lem}
\begin{proof} ~By (\ref{at}), we have $$a_{1} \le \alpha_0 a_0 + \beta_0  = \beta_0.$$
Proving by induction, assume that 
\begin{eqnarray} \label{atl}
a_{t} \le \beta_{t-1} + \sum_{j=1}^{t-1}\prod_{i=j}^{t -1}\alpha_i\beta_{j-1},
\end{eqnarray} 
then we have 
\begin{eqnarray*} 
a_{t+1} &\le& \alpha_t a_t + \beta_t  \quad \text{(by (\ref{at}))}\\
	&\le&  \beta_t + \alpha_t(\beta_{t-1} + \sum_{j=1}^{t-1}\prod_{i=j}^{t -1}\alpha_i\beta_{j-1})  \quad \text{(by (\ref{atl}))}\\
	& = &   \beta_t + \alpha_t\beta_{t-1} + \sum_{j=1}^{t-1}\prod_{i=j}^{t}\alpha_i\beta_{j-1}\\
	& = &   \beta_t + \sum_{j=1}^{t}\prod_{i=j}^{t }\alpha_i\beta_{j-1},
\end{eqnarray*} 
which ends the proof.
\end{proof}
\section{The issue in Zheng et al. [1]}\label{counterexample}
In order to prove the convergence theorem for Algorithm \ref{dist-EF-SGD}, Zheng et al. \cite{Zhengetal} have used the following lemma.

\medskip
\noindent
{\bf Lemma A (Lemma 2 of \cite{Zhengetal}, incorrect)}
{\it For any $t\ge 0$, $\tilde e_t, e_{t, i}, \eta_t$ from Algorithm \ref{dist-EF-SGD}, compressor parameter $\delta$ at (\ref{comp}), and gradient bound $G$ at (\ref{for_gti}),
\begin{eqnarray*} 
{\mathbf E}\left[\left\lVert \tilde e_t + \frac{1}{M}\sum_{i=1}^{M}e_{t,i}\right\rVert^2\right] 
\le \frac{8(1-\delta)G^2 }{\delta^2} \left(1 + \frac{16}{\delta^2}\right).
\end{eqnarray*}
}
Intuitively, Lemma A can become incorrect because its right-hand side only depends on the gradient bound $G$ and compressor parameter $\delta$, and {\it does not} capture the scaling factor $\eta_{t-1}/\eta_t$ of the errors $\tilde e_t$  and  $e_{t,i}$ of Algorithm \ref{dist-EF-SGD}. More formally, the following claim states that Lemma A is invalidated when the learning rate sequence $\{\eta_t\}$ is decreasing.

\begin{claim}\label{counterthm} Lemma A (i.e., Lemma 2 of \cite{Zhengetal}) does not hold. More precisely, referring to Algorithm \ref{dist-EF-SGD}, there exist a sequence of loss functions $\mathcal L(x_t, \xi)$, a decreasing sequence $\{\eta_t\}_{t\ge -1}$, a number $\delta$ with respect to a compressor $\mathcal C$, and a step $t$ such that 
\begin{eqnarray} \label{wrongineq}
{\mathbf E}\left[\left\lVert \tilde e_t \!+\! \frac{1}{M}\sum_{i=1}^{M}e_{t,i}\right\rVert^2\right] 
\!>\! \frac{8(1-\delta)G^2 }{\delta^2} \left(1 \!+\! \frac{16}{\delta^2}\right).
\end{eqnarray}
\end{claim}
Claim \ref{counterthm} is justified by the following counter-examples, in which we intentionally  utilize the fact that the quotient  $\eta_{t-1}/\eta_t$ as in line \ref{p_t_i_} of Algorithm \ref{dist-EF-SGD} can be large with {\it decreasing} learning rate sequences.
%%%%%%%%%%%%%%%%%%

\begin{ex}(Convex case)\label{linear} For $t\ge 0$ and $x_t, \xi\in \bbR$, we consider the sequence of loss functions 
$${\mathcal L}(x_t, \xi) = \varphi(x_t) =\frac{1}{4}x_t$$
in the constraint set $ [-1,1]$, the decreasing sequence of learning rate $\{\eta_t\}_{t\ge -1}$ with 
$$\eta_{-1} = 0, \left\{\eta_t =\frac{1}{48t+2}\right\}_{t\ge0},$$
the compressor $\mathcal C : \mathbb R\to \mathbb R$ such that $\forall x\in \mathbb R$
$$\mathcal C(x) = \frac{x}{0.77}.$$ 
Then at $t = 1$, Claim \ref{counterthm} holds true.
\end{ex}

\begin{proof}({\bf (Justification of Counter-example \ref{linear})}
It is trivial that the loss function ${\mathcal L}$ satisfies all the Assumptions 1, 2, and 3. The upper bound gradient of $f$ is $G = \frac{1}{4}$ because we have $g_{t,i} = \frac{1}{4} \forall t, i$. 

$\bullet$ The function $\mathcal C$ with $\mathcal C(x) = \frac{x}{0.77}$ is a compressor with respect to $\delta = 0.9$. Indeed, we have 
\begin{align*}
&\lVert \mathcal C(x) -x\rVert^2 \le (1-\delta)\lVert x\rVert^2\\
 &\Leftrightarrow \left\lVert \frac{x}{0.77}-x\right\rVert^2 \le (1-\delta)\lVert x\rVert^2 \\
 &\Leftrightarrow \left\lVert \frac{1}{0.77}-1\right\rVert^2 \le 1-\delta.
%  \Leftrightarrow&& \delta \le 1 - \left\lVert \frac{1}{0.77}-1\right\rVert^2  = 0.9107775341541575.
\end{align*}
The last inequality is equivalent to $$\delta \le 1 - \left\lVert \frac{1}{0.77}-1\right\rVert^2  = 0.9107775341541.$$ Therefore $\delta = 0.9$ suffices.

 To continue, let us consider the number of workers is $M = 2$.
Initially $e_{0, i} = 0$ on each worker $i\in\{1,2\}$ and $\tilde e_0 = 0$ on server.
Because the stochastic gradients are the same on each worker, the results of computations on each worker are the same. So it is sufficient to consider the computations on worker 1 in details. 

$\bullet$ At $t=0$ we have the computations on the workers and the server as follows.

-- On worker 1:
\begin{eqnarray*} 
p_{0,1} &=& g_{0,1} + \frac{\eta_{-1}}{\eta_0}e_{0,1}\\
	& = & g_{0,1} = \frac{1}{4},\\
\Delta_{0,1} &=& \mathcal C(p_{0,1}) \\
	&&= \frac{p_{0,1}}{0.77} = 0.3246753246753,\\
e_{1,1} & = & p_{0,1} - \Delta_{0,1} =-0.07467532467532.
\end{eqnarray*} 

-- On worker 2,  $p_{0,2} = p_{0,1}$, $\Delta_{0,2} = \Delta_{0,1}$, and $e_{1,2} = e_{1,1}$.

-- On server:
\begin{eqnarray*} 
\tilde p_{0} &=& \frac{1}{2}(\Delta_{0,1} + \Delta_{0,2}) + \frac{\eta_{-1}}{\eta_0}\tilde e_{0}\\ 
	& = & \Delta_{0,1}  = 0.3246753246753, \\
\tilde{\Delta}_0 & = & \mathcal C(\tilde p_{0})\\
	& = & \frac{\tilde p_{0}}{0.77} = 0.42165626581210,\\
\tilde e_{1} & = & \tilde p_{0} - \tilde{\Delta}_0 = -0.09698094113678.
\end{eqnarray*} 

$\bullet$ At $t=1$ we have the computations on the workers and the server as follows.

-- On worker 1:
\begin{eqnarray*} 
p_{1,1} &=& g_{1,1} + \frac{\eta_{0}}{\eta_1}e_{1,1}\\
	%& = & \frac{1}{4} + \frac{50}{2}(-0.07467532467532) \\
	&=& -1.6168831168831,\\
\Delta_{1,1} &=& \mathcal C(p_{1,1}) \\
	&&= \frac{p_{1,1}}{0.77} = -2.0998482037443,\\
e_{2,1} & = & p_{1,1} - \Delta_{1,1} =0.48296508686119.
\end{eqnarray*} 

-- On worker 2: $p_{1,2} = p_{1,1}$, $\Delta_{1,2}=\Delta_{1,1}$, and $e_{2,2} = e_{2,1}$.

-- On server:
\begin{eqnarray*} 
\tilde p_{1} &=& \frac{1}{2}(\Delta_{1,1} + \Delta_{1,2}) + \frac{\eta_{0}}{\eta_1}\tilde e_{1}\\ 
	& = & \Delta_{1,1} + \frac{\eta_{0}}{\eta_1}\tilde e_{1}\\
	%& = & -2.0998482037443 \\
	%&&+ \frac{50}{2}(-0.09698094113678)\\
	& = &-4.52437173216,\\
\tilde{\Delta}_1 & = & \mathcal C(\tilde p_{1})\\
	& = & \frac{\tilde p_{1}}{0.77} = -5.8758074443687,\\
\tilde e_{2} & = & \tilde p_{1} - \tilde{\Delta}_1 = 1.3514357122048.
\end{eqnarray*} 
Now we compute the left and right hand sides of (\ref{wrongineq}) with $t=2$. We have 
\begin{eqnarray*} 
\lefteqn{\left\lVert \tilde e_{2} + \frac{1}{2}(e_{2,1} + e_{2,2})\right\rVert^2}\\
& \!=\! & \left\lVert \tilde e_{2} + e_{2,1}\right\rVert^2\\
%& \!=\! & (1.3514357122048 +0.48296508686119)^2\\
&\!=\!& {\color{black}3.365026291613992}%3.3805310848189216
\end{eqnarray*} 
and 
\begin{eqnarray*}
\frac{8(1\!-\!\delta)G^2 }{\delta^2} \left(1\! +\! \frac{16}{\delta^2}\right) 
& \!=\! & \frac{8(0.1)\left(\frac{1}{4}\right)^2 }{0.9^2} \left(1 \!+\! \frac{16}{0.9^2}\right)\\
&\!=\!& 1.2810547172687086.
\end{eqnarray*} 
Thus 
$$\left\lVert \tilde e_{2} + \frac{1}{2}(e_{2,1} + e_{2,2})\right\rVert^2 > \frac{8(1-\delta)G^2 }{\delta^2} \left(1 + \frac{16}{\delta^2}\right),$$
and then Claim \ref{counterthm} follows.
\end{proof}

\begin{ex}(Convex case)\label{x^2} For $t\ge 0$ and $x_t, \xi\in \bbR$, we consider the sequence of loss functions 
$${\mathcal L}(x_t, \xi) = \varphi(x_t) =x_t^2$$
in the constraint set $ [-1,1]$, the decreasing sequence of learning rate $\{\eta_t\}_{t\ge -1}$ with 
$$\eta_{-1} = 0, \left\{\eta_t = \frac{3}{4}\left(\frac{1}{26t+2}\right)\right\}_{t\ge0},$$
and the  following compressor $\mathcal C : \mathbb R\to \mathbb R$ with parameter $\delta = 0.9$ as in Counter-example \ref{linear},
$$\mathcal C(x) = \frac{x}{0.77}.$$ 
Then at $t = 1$, Claim \ref{counterthm} holds true.
\end{ex}

\begin{proof}({\bf (Justification of Counter-example \ref{x^2})}
It is trivial that the loss function ${\mathcal L}$ is $2$-smooth and ${\mathcal L}$ satisfies all the Assumptions 1, 2, and 3. Since $\nabla f(x) = 2x$ and $g_{t,i} = 2x_{t}, \forall t, i$, we have the upper bound gradient of $f$ in the constraint set $ [-1,1]$ {\color{black}as} $G = 2$. Let us consider the number of workers $M = 2$.
Initially $e_{0, i} = 0$ on each worker $i\in\{1,2\}$ and $\tilde e_0 = 0$ on server. Let us take $x_0 = 1$.
Because the stochastic gradients are the same on each worker, the results of computations on each worker are the same. So it is sufficient to consider the computations on worker 1 in details.  

$\bullet$ At $t=0$ we have the computations on the workers and the server as follows.

-- On worker 1:
\begin{eqnarray*} 
p_{0,1} &=& g_{0,1} + \frac{\eta_{-1}}{\eta_0}e_{0,1}\\
	& = & g_{0,1} =2,\\
\Delta_{0,1} &=& \mathcal C(p_{0,1}) \\
	&&= \frac{p_{0,1}}{0.77} = 2.5974025974025974,\\
e_{1,1} & = & p_{0,1} - \Delta_{0,1} =-0.5974025974025974.
\end{eqnarray*} 

-- On worker 2,  $p_{0,2} = p_{0,1}$, $\Delta_{0,2} = \Delta_{0,1}$, and $e_{1,2} = e_{1,1}$.

-- On server:
\begin{eqnarray*} 
\tilde p_{0} &=& \frac{1}{2}(\Delta_{0,1} + \Delta_{0,2}) + \frac{\eta_{-1}}{\eta_0}\tilde e_{0}\\ 
	& = & \Delta_{0,1}  = 2.5974025974025974, \\
\tilde{\Delta}_0 & = & \mathcal C(\tilde p_{0})\\
	& = & \frac{\tilde p_{0}}{0.77} = 3.3732501264968797,\\
x_{1} &=& x_0 - \eta_0 \tilde{\Delta}_0\\
%& = & 1 - \frac{3}{8}(3.3732501264968797)\\
&=& -0.26496879743632995,\\
\tilde e_{1} & = & \tilde p_{0} - \tilde{\Delta}_0 = -0.7758475290942823.
\end{eqnarray*} 

$\bullet$ At $t=1$ we have the computations on the workers and the server as follows.

-- On worker 1:
\begin{eqnarray*} 
g_{1,1} &=& \nabla f(x_1) = -0.5299375948726599,\\
p_{1,1} &=& g_{1,1} + \frac{\eta_{0}}{\eta_1}e_{1,1}\\
	%& = &  -0.5299375948726599 \\
	%&&+ \frac{28}{2}(-0.5974025974025974) \\
	& = & -8.893573958509023,\\
\Delta_{1,1} &=& \mathcal C(p_{1,1}) \\
	&=& \frac{p_{1,1}}{0.77} = -11.550096050011717,\\
e_{2,1} & = & p_{1,1} - \Delta_{1,1} =2.656522091502694.
\end{eqnarray*} 

-- On worker 2, $p_{1,2} = p_{1,1}$, $\Delta_{1,2}=\Delta_{1,1}$, and $e_{2,2} = e_{2,1}$.

-- On server:
\begin{eqnarray*} 
\tilde p_{1} &=& \frac{1}{2}(\Delta_{1,1} + \Delta_{1,2}) + \frac{\eta_{0}}{\eta_1}\tilde e_{1}\\ 
	& = & \Delta_{1,1} + \frac{\eta_{0}}{\eta_1}\tilde e_{1}\\
	%& = & -11.550096050011717\\
	%&& + \frac{28}{2}(-0.7758475290942823)\\
	& = &-22.41196145733167, \\
\tilde{\Delta}_1 & = & \mathcal C(\tilde p_{1})\\
	& = & \frac{\tilde p_{1}}{0.77} = -29.106443451080093,\\
\tilde e_{2} & = & \tilde p_{1} - \tilde{\Delta}_1 = 6.694481993748422.
\end{eqnarray*} 
Now we compute the left and right hand sides of (\ref{wrongineq}) with $t=2$. We have 
\begin{eqnarray*} 
\lefteqn{\left\lVert \tilde e_{2} + \frac{1}{2}(e_{2,1} + e_{2,2})\right\rVert^2}\\
&\!=\!&\left\lVert \tilde e_{2} + e_{2,1} \right\rVert^2\\
%& \!=\! & (6.694481993748422 + 2.656522091502694)^2\\
 &\!=\!& 87.44127740238307
\end{eqnarray*} 
and 
\begin{eqnarray*}
\frac{8(1-\delta)G^2 }{\delta^2} \left(1 + \frac{16}{\delta^2}\right) 
& = & \frac{8(0.1)2^2 }{0.9^2} \left(1 + \frac{16}{0.9^2}\right)\\
&=& 81.98750190519735.
\end{eqnarray*} 
Thus 
$$\left\lVert \tilde e_{2} + \frac{1}{2}(e_{2,1} + e_{2,2})\right\rVert^2 > \frac{8(1-\delta)G^2 }{\delta^2} \left(1 + \frac{16}{\delta^2}\right),$$
and then Claim \ref{counterthm} follows.
\end{proof}

\begin{ex}(Non-convex case)\label{exnonconvex} For $t\ge 0, x_t, \xi \in \mathbb R$, we consider the sequence of loss functions 
$${\mathcal L}(x_t, \xi) = \varphi(x_t) =\frac{1}{1 + e^{-x_t}},$$
the decreasing sequence of learning rate $\{\eta_t\}_{t\ge -1}$ with 
$$\eta_{-1} = 0, \left\{\eta_t = \frac{3}{2}\left(\frac{1}{48t+2}\right)\right\}_{t\ge0},$$
and the following compressor $\mathcal C : \mathbb R\to \mathbb R$ with parameter $\delta = 0.9$ as in Counter-example \ref{linear},
$$\mathcal C(x) = \frac{x}{0.77}.$$ 
Then at $t = 1$, Claim \ref{counterthm} holds true.
\end{ex}

\begin{proof}({\bf (Justification of Counter-example \ref{exnonconvex})} First, we check that  Assumptions 1-3 are satisfied. 

$\bullet$ {\color{black}The function}  $\varphi$ is lower-bounded because $0\le \varphi(x) \le 1, \forall x\in \bbR$. The upper bound of $\nabla\varphi(x)$ is $G = \frac{1}{4}$, since
\begin{eqnarray}\label{upperboundsigmoid} 
\nabla \varphi(x) \le \frac{1}{4} &\Leftrightarrow& \varphi(x) (1-\varphi(x)) \le \frac{1}{4} \\
&\Leftrightarrow& -\varphi(x)^2 + \varphi(x) -\frac{1}{4} \le 0\nonumber\\
&\Leftrightarrow& -\left(\varphi(x) -\frac{1}{2}\right)^2 \le 0,\nonumber
\end{eqnarray}
which {\color{black}holds true} for all $x\in \bbR$.

$\bullet$ {\color{black}The function} $\varphi$ is $L$-smooth, with $L=1$. Indeed, for all $x, y \in \mathbb R$, we have 
\begin{eqnarray*}
\lefteqn{|\nabla\varphi(x) - \nabla\varphi(y)|}\\
&=& | \varphi(x)(1-\varphi(x)) - \varphi(y)(1-\varphi(y))| \\
& = & | \varphi(x)  - \varphi(y) + (\varphi(y)-\varphi(x))(\varphi(y)+\varphi(x)) |\\
& = & | [\varphi(x)  - \varphi(y)][1- (\varphi(x)+\varphi(y))]| .
\end{eqnarray*}
Since $0\le \varphi(\xi) \le 1(\forall \xi)$, we have $0\le \varphi(x)+\varphi(y) \le 2$. Therefore $-1\le 1- (\varphi(x)+\varphi(y))\le 1$, and hence 
$$| [\varphi(x)  - \varphi(y)][1- (\varphi(x)+\varphi(y))]| \le |\varphi(x)  - \varphi(y)|.$$
This means that in order to prove $|\nabla\varphi(x) - \nabla\varphi(y)| \le |x-y|$, it is sufficient to prove 
\begin{eqnarray} \label{decre}
|\varphi(x)  - \varphi(y)| \le |x-y|.
\end{eqnarray} 
If $x\ge y$, we obtain $\varphi(x) \ge \varphi(y)$. Therefore
\begin{eqnarray*} 
(\ref{decre}) 
&\Leftrightarrow& \varphi(x)  - \varphi(y) \le x-y\\
&\Leftrightarrow& \varphi(x)  - x \le \varphi(y)-y.
\end{eqnarray*}
Let $\phi(x) =  \varphi(x)  - x$. Because $\nabla \phi(x) = \nabla \varphi(x) - 1 \le \frac{1}{4} - 1<0$ by (\ref{upperboundsigmoid}), we have $\phi(x)$ is a decreasing function. Therefore the inequality $\varphi(x)  - x \le \varphi(y)-y$ holds true, and hence (\ref{decre}) is proven. By the same technique, we obtain (\ref{decre}) for the case  $x< y$.

To continue, let us consider the number of workers $M = 2$.
We initialize $e_{0, i} = 0$ on each worker $i\in\{1,2\}$ and $\tilde e_0 = 0$ on server. 
Let us take $x_0 = 0$. 
Because the stochastic gradients are the same on each worker, the results of computations on each worker are the same. So it is sufficient to consider the computations on worker 1 in details. We have 
\begin{eqnarray*}  
g_{0,1} = g_{0,2} =  \nabla\varphi(x_0) = 0.25.
\end{eqnarray*} 
$\bullet$ At $t=0$ we have the computations on the workers and the server as follows.

-- On worker 1:
\begin{eqnarray*} 
p_{0,1} &=& g_{0,1} + \frac{\eta_{-1}}{\eta_0}e_{0,1}\\
	& = & g_{0,1} = 0.25,\\
\Delta_{0,1} &=& \mathcal C(p_{0,1}) \\
	&=& \frac{p_{0,1}}{0.77} = 0.3246753246753247,\\
e_{1,1} & = & p_{0,1} - \Delta_{0,1} =-0.07467532467532467.
\end{eqnarray*} 

-- On worker 2,  $p_{0,2} = p_{0,1}$, $\Delta_{0,2} = \Delta_{0,1}$, and $e_{1,2} = e_{1,1}$.

-- On server:
\begin{eqnarray*} 
\tilde p_{0} &=& \frac{1}{2}(\Delta_{0,1} + \Delta_{0,2}) + \frac{\eta_{-1}}{\eta_0}\tilde e_{0}\\ 
	& = &\Delta_{0,1} = 0.3246753246753247, \\
\tilde{\Delta}_0 & = & \mathcal C(\tilde p_{0})\\
	& = & \frac{\tilde p_{0}}{0.77} = 0.42165626581210996,\\
x_{1} &=& x_0 - \eta_0 \tilde{\Delta}_0 \\
	%& = & 0 - \frac{3}{4}( 0.42165626581210996)\\
	& =& -0.3162421993590825,\\
\tilde e_{1} & = & \tilde p_{0} - \tilde{\Delta}_0 = -0.09698094113678529.
\end{eqnarray*} 
$\bullet$ At $t=1$ we have the computations on the workers and the server as follows.

-- On worker 1:
\begin{eqnarray*} 
g_{1,1} & = &  \nabla\varphi(x_1) = 0.243852158038919\\
p_{1,1} &=& g_{1,1} + \frac{\eta_{0}}{\eta_1}e_{1,1}\\
	%& = & 0.243852158038919 \\
	%&&+ \frac{50}{2}(-0.07467532467532467) \\
	&=& -1.6230309588441978,\\
\Delta_{1,1} &=& \mathcal C(p_{1,1}) \\
	&&= \frac{p_{1,1}}{0.77} = -2.1078324140833735,\\
e_{2,1} & = & p_{1,1} - \Delta_{1,1} =0.4848014552391757.
\end{eqnarray*} 

-- On worker 2, $p_{1,2} = p_{1,1}$, $\Delta_{1,2}=\Delta_{1,1}$, and $e_{2,2} = e_{2,1}$.

-- On server:
\begin{eqnarray*} 
\tilde p_{1} &=& \frac{1}{2}(\Delta_{1,1} + \Delta_{1,2}) + \frac{\eta_{0}}{\eta_1}\tilde e_{1}\\ 
	& = & \Delta_{1,1} + \frac{\eta_{0}}{\eta_1}\tilde e_{1}\\
	%& = & -2.1078324140833735 \\
	%&& + \frac{50}{2}(-0.09698094113678529)\\
	& = &-4.532355942503006,\\
\tilde{\Delta}_1 & = & \mathcal C(\tilde p_{1})\\
	& = & \frac{\tilde p_{1}}{0.77} = -5.886176548705203,\\
\tilde e_{2} & = & \tilde p_{1} - \tilde{\Delta}_1 = 1.3538206062021967.
\end{eqnarray*} 

Now we compute the left and right hand sides of (\ref{wrongineq}) with $t=2$. We have 
\begin{align*} 
\left\lVert \tilde e_{2} + \frac{1}{2}(e_{2,1} + e_{2,2})\right\rVert^2
& \!=\!  \left\lVert \tilde e_{2} + e_{2,1} \right\rVert^2\\
%& \!=\!  (1.3538206062021967 + 0.4848014552391757)^2\\
 &\!=\! 3.3805310848189216
\end{align*} 
and, with $\delta = 0.9$, 
\begin{eqnarray*}
\frac{8(1-\delta)G^2 }{\delta^2} \left(1 + \frac{16}{\delta^2}\right) 
&\! =\! & \frac{8(0.1)\left(\frac{1}{4}\right)^2 }{0.9^2} \left(1 + \frac{16}{0.9^2}\right)\\
&\!=\!& 1.2810547172687086.
\end{eqnarray*} 
Thus 
$$\left\lVert \tilde e_{2} + \frac{1}{2}(e_{2,1} + e_{2,2})\right\rVert^2 > \frac{8(1-\delta)G^2 }{\delta^2} \left(1 + \frac{16}{\delta^2}\right),$$
and hence Claim \ref{counterthm} follows.
\end{proof}

\section{Correcting the error bound of Zheng et al. [1]}\label{correctingerror}
In general, the error ${\mathbf E}\left[\left \lVert \tilde e_{t+1} + \frac{1}{M}\sum_{i=1}^{M}e_{t+1,i}\right\rVert^2\right] $ is bounded as follows.

\begin{thm}({Fix for Lemma 2 of \cite{Zhengetal}})\label{error} With $\tilde e_t, e_{t, i},  \eta_t$, and $\delta$ from Algorithm \ref{dist-EF-SGD}, for arbitrary $\{\eta_t\}$, we have
\begin{eqnarray*}
\lefteqn{{\mathbf E}\left[\left \lVert \tilde e_{t+1} + \frac{1}{M}\sum_{i=1}^{M}e_{t+1,i}\right\rVert^2\right] }\\
&\le& \frac{2(1-\delta)(2-\delta)G^2}{\delta}\sum_{k=0}^{t}\frac{\eta_{t-k}^2}{\eta_t^2}\alpha^k\\
&&+ \frac{4(1-\delta)(2-\delta)^3G^2}{\delta^2}\sum_{j=0}^{t} \alpha^{t-j}\sum_{k=0}^{j}\frac{\eta_{j-k}^2}{\eta_{t}^2}\alpha^k,
\end{eqnarray*}
where $\alpha = 1-\frac{\delta}{2}$ and gradient bound $G$ is at (\ref{for_gti}).
\end{thm}

\begin{rem}[Sanity check of the new upper bound] The right-hand side of Theorem \ref{error} can become large together with decreasing leaning rate sequences. Therefore, the error bounds of the sequences in Counter-examples 1-3 do satisfy Theorem \ref{error}. Indeed, the upper bound on the error in Theorem \ref{error} at $t=1$ is
\begin{eqnarray*}
\lefteqn{U =  \frac{2(2-\delta)(1-\delta)G^2}{\delta} \left( 1 + \frac{\eta_0^2}{\eta_1^2}\left( 1-\frac{\delta}{2}\right)\right)} \\
&&+ \frac{4(1-\delta)(2-\delta)^3G^2}{\delta^2} \left( 1 +2 \frac{\eta_0^2}{\eta_1^2}\left( 1-\frac{\delta}{2}\right)\right).
\end{eqnarray*}
\noindent
Concretely, at sanity check,

\noindent
$\bullet$ in Counter-example \ref{linear}: $$U = 33.550763888888895$$
 which is indeed larger than 
\begin{eqnarray*}
\left \lVert \tilde e_{2} + \frac{1}{M}\sum_{i=1}^{M}e_{2,i}\right\rVert^2 = 3.365026291613992.
\end{eqnarray*}
$\bullet$ in Counter-example \ref{x^2}: $$U = 675.8530370370372$$ 
which is indeed larger than 
\begin{eqnarray*}
\left \lVert \tilde e_{2} + \frac{1}{M}\sum_{i=1}^{M}e_{2,i}\right\rVert^2 = 87.44127740238307.
\end{eqnarray*}
$\bullet$ in Counter-example \ref{exnonconvex}: 
\begin{eqnarray*}
U = 33.550763888888895
\end{eqnarray*}
which is indeed larger than 
\begin{eqnarray*}
\left \lVert \tilde e_{2} + \frac{1}{M}\sum_{i=1}^{M}e_{2,i}\right\rVert^2 = 3.3805310848189216.
\end{eqnarray*}
\end{rem}

\begin{proof}~[{\bf Proof of Theorem \ref{error}}] We have
\begin{eqnarray} \label{sp}
\lefteqn{{\mathbf E}\left[\left \lVert \tilde e_{t+1} + \frac{1}{M}\sum_{i=1}^{M}e_{t+1,i}\right\rVert^2\right] }\nonumber\\
&\le& 2{\mathbf E}\left[\lVert \tilde e_{t+1}\rVert^2\right] + 2 {\mathbf E}\left[\left\lVert\frac{1}{M}\sum_{i=1}^{M}e_{t+1,i}\right\rVert^2\right]\nonumber\\
&\le& 2{\mathbf E}[\lVert \tilde e_{t+1}\rVert^2] + \frac{2}{M} \sum_{i=1}^{M}{\mathbf E}[\lVert e_{t+1,i}\rVert^2],
\end{eqnarray} 
where the first inequality is by the fact that $(a+b)^2 \le 2a^2 + 2b^2, \forall a,b$, and the second inequality is by Lemma \ref{normineq}. We will separately bound the two terms of (\ref{sp}). 
Firstly, we consider $\frac{1}{M}\sum_{i=1}^{M}{\mathbf E}[\lVert e_{t+1,i}\rVert^2] $. We have
\begin{eqnarray}
\lefteqn{\frac{1}{M}\sum_{i=1}^{M}{\mathbf E}[\lVert e_{t+1,i}\rVert^2]
{=}  \frac{1}{M}\sum_{i=1}^{M}{\mathbf E}[\lVert \mathcal C(p_{t,i})-p_{t,i}\rVert^2}\label{pt1}\\
&{\le} & \frac{1-\delta}{M}\sum_{i=1}^{M}{\mathbf E}[\lVert p_{t,i}\rVert^2]\label{pt2}\\  
&{=} & \frac{1-\delta}{M}\sum_{i=1}^{M}{\mathbf E}\left[\left\lVert g_{t,i} + \frac{\eta_{t-1}}{\eta_t}e_{t,i}\right\rVert^2\right]\label{pt3}\\
&{\le} & \frac{(1-\delta)(1+\gamma)}{M}\sum_{i=1}^{M}{\mathbf E}\left[\left\lVert \frac{\eta_{t-1}}{\eta_t}e_{t,i}\right\rVert^2\right] \nonumber\\
&&+ \frac{(1-\delta)(1+1/\gamma)}{M}\sum_{i=1}^{M}{\mathbf E}\left[\lVert g_{t,i}\rVert^2\right]\label{pt4}\\
&  {\le} &(1-\delta)(1+\gamma)\frac{\eta_{t-1}^2}{\eta_t^2}\left(\frac{1}{M}\sum_{i=1}^{M}{\mathbf E}\left[\lVert e_{t,i}\rVert^2\right]\right) \nonumber\\
&&+ (1-\delta)(1+1/\gamma)G^2, \label{pt5}
\end{eqnarray}
where $(\ref{pt1})$ and $(\ref{pt3})$ is by the setting of $e_{t+1,i}$ and $p_{t,i}$ in Algorithm \ref{dist-EF-SGD}, $(\ref{pt2})$ is by the definition of compressor $\mathcal C$, $(\ref{pt4})$ is by Young inequality with any $\gamma>0$, and $(\ref{pt5})$ is by (\ref{for_gti}). 
Now, for all $t\ge 0$, applying Lemma \ref{recseq} to the inequality $(\ref{pt5})$ with 
\begin{eqnarray*}
a_{t+1} &=& \frac{1}{M}\sum_{i=1}^{M}{\mathbf E}[\lVert e_{t+1,i}\rVert^2],\\
\alpha_t &=& (1-\delta)(1+\gamma)\frac{\eta_{t-1}^2}{\eta_t^2},\\
\beta & = &(1-\delta)(1+1/\gamma)G^2,
\end{eqnarray*}
we have
$$\frac{1}{M}\sum_{i=1}^{M}{\mathbf E}[\lVert e_{t+1,i}\rVert^2] \le \beta\left(1 +\sum_{j=1}^t\prod_{i=j}^t\alpha_i \right)$$
Moreover, since
\begin{eqnarray*}
\lefteqn{1 + \sum_{j=1}^t\prod_{i=j}^t\alpha_i } \\
& = & 1 + \sum_{j=1}^{t} \prod_{i=j}^{t}(1-\delta)(1+\gamma)\frac{\eta_{i-1}^2}{\eta_i^2}\\
& = & 1 + \sum_{j=1}^{t}\frac{\eta_{j-1}^2}{\eta_t^2}[(1-\delta)(1+\gamma)]^{t-(j-1)} \\
& = & \sum_{j=1}^{t+1}\frac{\eta_{j-1}^2}{\eta_t^2}[(1-\delta)(1+\gamma)]^{t-(j-1)}\\
& = & \sum_{k=0}^{t}\frac{\eta_{t-k}^2}{\eta_t^2}\left[(1-\delta)(1+\gamma)\right]^k,
\end{eqnarray*}
we obtain 
\begin{eqnarray*}
\lefteqn{ \frac{1}{M}\sum_{i=1}^{M}{\mathbf E}[\lVert e_{t+1,i}\rVert^2]}\\
& \le& (1-\delta)(1+1/\gamma)G^2\sum_{k=0}^{t}\frac{\eta_{t-k}^2}{\eta_t^2}\left[(1-\delta)(1+\gamma)\right]^k.
\end{eqnarray*}
By choosing $\gamma = \frac{\delta}{2(1-\delta)}$, we have 
$$(1-\delta)(1+1/\gamma) = \frac{(1-\delta)(2-\delta)}{\delta}$$ 
and $$(1-\delta)(1+\gamma) = 1 - \frac{\delta}{2}.$$ 
Therefore 
\begin{eqnarray} \label{pti2}
\lefteqn{\frac{1}{M}\sum_{i=1}^{M}{\mathbf E}[\lVert e_{t+1,i}\rVert^2]}\nonumber\\
& \le&  \frac{(2-\delta)(1-\delta)G^2}{\delta}\sum_{k=0}^{t}\frac{\eta_{t-k}^2}{\eta_t^2}\alpha^k,
\end{eqnarray}
where $\alpha = 1 - \frac{\delta}{2}$.
Next, we consider the term ${\mathbf E}[\lVert \tilde e_{t+1}\rVert^2]$ of (\ref{sp}). By the setting of $e_{t+1,i}$ and $p_{t,i}$ in Algorithm \ref{dist-EF-SGD} and the definition of compressor $\mathcal C$, we have
\begin{eqnarray*}   
\lefteqn{{\mathbf E}[\lVert \tilde e_{t+1}\rVert^2]}\\
&=& {\mathbf E}[\lVert \mathcal C(\tilde p_{t}) - \tilde p_t\rVert^2] \\
& \le& (1-\delta){\mathbf E}[\lVert \tilde p_t\rVert^2]\\
& = & (1-\delta){\mathbf E}\left[\left \lVert \frac{1}{M}\sum_{i=1}^{M}\mathcal C(p_{t,i}) + \frac{\eta_{t-1}}{\eta_t}\tilde e_t\right\rVert^2\right]\\ 
&\le& (1-\delta)(1+\gamma) \frac{\eta_{t-1}^2}{\eta_t^2}{\mathbf E}[\lVert \tilde e_t\rVert^2] \\
&&+ (1\!-\!\delta)(1\!+\!1/\gamma){\mathbf E}\!\left[\left \lVert \frac{1}{M}\sum_{i=1}^{M}\mathcal C(p_{t,i}) \right\rVert^2\right],
\end{eqnarray*} 
where the last inequality is by Young inequality for any $\gamma>0$.
Looking at ${\mathbf E}\left[\left \lVert \frac{1}{M}\sum_{i=1}^{M}\mathcal C(p_{t,i}) \right\rVert^2\right]$, we have
\begin{align}
\lefteqn{{\mathbf E}\left[\left \lVert \frac{1}{M}\sum_{i=1}^{M}\mathcal C(p_{t,i}) \right\rVert^2\right]}\nonumber\\
&{\le} \frac{1}{M}\sum_{i=1}^{M} {\mathbf E}[\lVert \mathcal C(p_{t,i}) \rVert^2]\label{pt6}\\
&{\le}   \frac{1}{M}\sum_{i=1}^{M} ( 2{\mathbf E}[\lVert \mathcal C(p_{t,i}) - p_{t,i} \rVert^2] + 2 {\mathbf E}[\lVert p_{t,i} \rVert^2] )\label{pt7} \\
&{\le}   \frac{1}{M}\sum_{i=1}^{M} ( 2(1-\delta){\mathbf E}[\lVert p_{t,i} \rVert^2] + 2 {\mathbf E}[\lVert p_{t,i} \rVert^2] )\label{pt8}\\
& =   2(2-\delta)\frac{1}{M}\sum_{i=1}^{M}{\mathbf E}[\lVert p_{t,i}\rVert^2],\nonumber
\end{align}
where $(\ref{pt6})$ is by Lemma \ref{normineq}, $(\ref{pt7})$ is by  the fact that $(a+b)^2 \le 2a^2 + 2b^2, \forall a,b$, $(\ref{pt8})$ is by the definition of compressor $\mathcal C$. Therefore 
\begin{eqnarray*}
{\mathbf E}[\lVert \tilde e_{t+1}\rVert^2] 
\le(1-\delta)(1+\gamma) \frac{\eta_{t-1}^2}{\eta_t^2}{\mathbf E}[\lVert \tilde e_t\rVert^2] \\
+ 2(2-\delta)(1-\delta)(1+1/\gamma)\frac{1}{M}\sum_{i=1}^{M}{\mathbf E}[\lVert p_{t,i}\rVert^2].
\end{eqnarray*}
%%%%%%%%%%
Moreover, {\color{black}(\ref{pt2}) and (\ref{pti2}) yield} 
$$\frac{1-\delta}{M}\sum_{i=1}^{M}{\mathbf E}[\lVert p_{t,i}\rVert^2] \le \frac{(2-\delta)(1-\delta)G^2}{\delta}\sum_{k=0}^{t}\frac{\eta_{t-k}^2}{\eta_t^2}\alpha^k.$$
Therefore 
\begin{eqnarray*}
\lefteqn{{\mathbf E}[\lVert \tilde e_{t+1}\rVert^2] }\\
&\le &(1-\delta)(1+\gamma) \frac{\eta_{t-1}^2}{\eta_t^2}{\mathbf E}[\lVert \tilde e_t\rVert^2]\\
&&+ \frac{2(2-\delta)^2(1-\delta)(1+1/\gamma)G^2}{\delta}\sum_{k=0}^{t}\frac{\eta_{t-k}^2}{\eta_t^2}\alpha^k.
\end{eqnarray*}
With $\gamma = \frac{\delta}{2(1-\delta)}$, since $(1-\delta)(1+1/\gamma) = \frac{(1-\delta)(2-\delta)}{\delta}$ and $(1-\delta)(1+\gamma) = 1 - \frac{\delta}{2} = \alpha$, we have 
\begin{eqnarray*}
{\mathbf E}[\lVert \tilde e_{t+1}\rVert^2] 
&\le& \alpha\frac{\eta_{t-1}^2}{\eta_t^2}{\mathbf E}[\lVert \tilde e_t\rVert^2] \\
&& + \frac{2(1-\delta)(2-\delta)^3G^2}{\delta^2}\sum_{k=0}^{t}\frac{\eta_{t-k}^2}{\eta_t^2}\alpha^k.
\end{eqnarray*}
By applying Lemma \ref{recseq} with 
\begin{eqnarray*}
a_{t} &=& {\mathbf E}[\lVert \tilde e_i\rVert^2],\\
\alpha_t &=& \alpha\frac{\eta_{t-1}^2}{\eta_t^2},\\
\beta_t &=& \frac{2(1-\delta)(2-\delta)^3G^2}{\delta^2}\sum_{k=0}^{t}\frac{\eta_{t-k}^2}{\eta_t^2}\alpha^k,
\end{eqnarray*}
we obtain 
\begin{eqnarray*}
{\mathbf E}[\lVert \tilde e_{t+1}\rVert^2] 
&\le& \beta_t + \sum_{j=1}^t \left( \prod_{i=j}^t \alpha_i \beta_{j-1}\right).
\end{eqnarray*}
Since 
\begin{eqnarray*}
 \prod_{i=j}^t \alpha_i \beta_{j-1} 
 & = & \prod_{i=j}^t \alpha\frac{\eta_{i-1}^2}{\eta_i^2}\beta_{j-1}\\
 & = & \alpha^{t-(j-1)}\frac{\eta_{j-1}^2}{\eta_t^2}\beta_{j-1},
\end{eqnarray*}
 we have 
 \begin{eqnarray*}
 {\mathbf E}[\lVert \tilde e_{t+1}\rVert^2] &\le & \beta_t + \sum_{j=1}^t \alpha^{t-(j-1)}\frac{\eta_{j-1}^2}{\eta_t^2}\beta_{j-1}\\
 &= & \sum_{j=1}^{t+1}\alpha^{t-(j-1)} \frac{\eta_{j-1}^2}{\eta_t^2}\beta_{j-1}.
 \end{eqnarray*}
%Replacing $\beta_{j-1}\! = \!\frac{2(1-\delta)(2-\delta)^3G^2}{\delta^2}\sum_{k=0}^{j-1}\frac{\eta_{j-1-k}^2}{\eta_{j-1}^2}\alpha^k$ gives us
Therefore,
 \begin{flalign} \label{forme}  
\lefteqn{{\mathbf E}[\lVert \tilde e_{t+1}\rVert^2]}\nonumber\\
 &\le  \sum_{j=1}^{t+1} \alpha^{t-(j-1)}\frac{\eta_{j-1}^2}{\eta_t^2} \frac{2(1-\delta)(2-\delta)^3G^2}{\delta^2}\sum_{k=0}^{j-1}\frac{\eta_{j-1-k}^2}{\eta_{j-1}^2}\alpha^k\nonumber\\
 &=  \frac{2(1-\delta)(2-\delta)^3G^2}{\delta^2}\sum_{j=1}^{t+1} \alpha^{t-(j-1)} \sum_{k=0}^{j-1}\frac{\eta_{j-1-k}^2}{\eta_{t}^2}\alpha^k\nonumber\\
 &=  \frac{2(1-\delta)(2-\delta)^3G^2}{\delta^2}\sum_{j=0}^{t} \alpha^{t-j} \sum_{k=0}^{j}\frac{\eta_{j-k}^2}{\eta_{t}^2}\alpha^k.
\end{flalign}
Substituting (\ref{pti2}) and (\ref{forme}) to (\ref{sp}), we obtain 
\begin{eqnarray*}
\lefteqn{{\mathbf E}\left[\left \lVert \tilde e_{t+1} + \frac{1}{M}\sum_{i=1}^{M}e_{t+1,i}\right\rVert^2\right] }\\
&\le& \frac{2(2-\delta)(1-\delta)G^2}{\delta}\sum_{k=0}^{t}\frac{\eta_{t-k}^2}{\eta_t^2}\alpha^k\\
&&+ \frac{4(1-\delta)(2-\delta)^3G^2}{\delta^2}\sum_{j=0}^{t} \alpha^{t-j} \sum_{k=0}^{j}\frac{\eta_{j-k}^2}{\eta_{t}^2}\alpha^k,
\end{eqnarray*}
as claimed in Theorem \ref{error}.
\end{proof}

As a sanity check, Theorem \ref{error} matches the results given in \cite{Zhengetal} when the learning rate is non--decreasing. As a result, Theorem \ref{mainthmcase1} and Theorem A agree when the learning rate is non-decreasing.

\begin{cor}({Sanity check of Theorem \ref{error}, cf. Lemma 6 of \cite{Zhengetal} with $\mu = 0$}) \label{errorcor}  In Theorem \ref{error}, if $\{\eta_t\}$ is a non-decreasing sequence such that $\eta_t >0, \forall t \ge 0$, then 
\begin{eqnarray*}
{\mathbf E}\left[\left \lVert \tilde e_{t+1}\! +\! \frac{1}{M}\sum_{i=1}^{M}e_{t+1,i}\right\rVert^2\right] 
\le \frac{8(1-\delta)G^2}{\delta^2}\left(1\! +\! \frac{16}{\delta^2} \right).
\end{eqnarray*}
\end{cor}

\begin{proof} ~Since $\{\eta_t\}$ is non-decreasing, we have $\frac{\eta_{t-k}^2}{\eta_t^2} \le 1$. Moreover, since  $\alpha = 1-\frac{\delta}{2} \in (0,1)$, we obtain
\begin{eqnarray}\label{reduce1}
\sum_{k=0}^{t}\frac{\eta_{t-k}^2}{\eta_t^2}\alpha^k \le  \sum_{k=0}^{t}\alpha^k \le \frac{2}{\delta},
\end{eqnarray} 
and 
\begin{eqnarray*} \label{reduce2}
\sum_{j=0}^{t} \alpha^{t-j} \sum_{k=0}^{j}\frac{\eta_{j-k}^2}{\eta_{t}^2}\alpha^k 
&\le& \sum_{j=0}^{t} \alpha^{t-j} \sum_{k=0}^{j}\alpha^k\nonumber\\
&\le& \sum_{j=0}^{t} \alpha^{t-j}\left( \frac{2}{\delta}\right)\nonumber\\
&\le& \frac{4}{\delta^2}.
\end{eqnarray*}
Replacing (\ref{reduce1}) and (\ref{reduce2}) to Theorem \ref{error}, we have the result stated in Corollary \ref{errorcor}.
\end{proof}

\section{ Correcting the convergence theorem of Zheng et al. [1]}\label{corectmainthm}
{\color{black}Because} the error bound plays a crucial role in the proof of the convergence theorem of dis-EF-SGD, fixing  \cite[Lemma 2]{Zhengetal} as in Theorem \ref{error} leads to the consequence that the convergence theorem need to be fixed as well.   

\begin{proof}~({\bf Proof of Theorem \ref{mainthmcase1}})
Following \cite{Zhengetal}, we consider the {\color{black} iteration} 
$$\tilde x_t = x_t -\eta_{t-1}\left( \tilde e_t + \frac{1}{M}\sum_{i=1}^M e_{t,i}\right),$$ where $x_t, \tilde e_t$ and $e_{t, i}$ are generated from Algorithm \ref{dist-EF-SGD}. Then, by \cite[Lemma 1]{Zhengetal}, 
\begin{eqnarray}\label{lemofZheng} \tilde x_{t+1} = \tilde x_t - \eta_t\frac{1}{M}\sum_{i=1}^M g_{t,i}.
\end{eqnarray}
Since $f$ is $L$-smooth, by (\ref{Lippro}), we have
\begin{align} \label{firstineq}
{\mathbf E}_t[f(\tilde x_{t+1})] \nonumber\\
&\le  f(\tilde x_t) + \langle \nabla f(\tilde x_{t}), {\mathbf E}_t[\tilde x_{t+1} - \tilde x_{t}] \rangle \nonumber\\
&\quad+ \frac{L}{2}{\mathbf E}_t[\lVert \tilde x_{t+1} - \tilde x_{t}\rVert^2] \nonumber\\
& {=}   f(\tilde x_t) - \eta_t\left\langle \nabla f(\tilde x_{t}), {\mathbf E}_t\left[ \frac{1}{M}\sum_{i=1}^Mg_{t,i}\right] \right\rangle \nonumber\\
&\quad+ \frac{L\eta_t^2}{2}{\mathbf E}_t\left[\left\lVert\frac{1}{M}\sum_{i=1}^Mg_{t,i} \right\rVert^2\right],
\end{align}
Moreover, we have 
\begin{flalign} \label{expectform}
\lefteqn{{\mathbf E}_t\left[\left\lVert\frac{1}{M}\sum_{i=1}^Mg_{t,i} \right\rVert^2\right]}\nonumber\\
&\!=\! \lVert  \nabla f(x_t)\rVert^2 + {\mathbf E}_t\left[\left\lVert\frac{1}{M}\sum_{i=1}^Mg_{t,i} - \nabla f(x_t)\right\rVert^2\right],
\end{flalign}
which follows from the fact that ${\mathbf E}[\lVert X-{\mathbf E}[X]\rVert^2] = {\mathbf E}[\lVert X\rVert^2] - \lVert {\mathbf E}[X]\rVert^2$. Substituting (\ref{expectform}) to (\ref{firstineq}), we obtain 
\begin{eqnarray}
\lefteqn{{\mathbf E}_t[f(\tilde x_{t+1})] }\nonumber\\
&\le&   f(\tilde x_t) - \eta_t\left\langle \nabla f(\tilde x_{t}), \nabla f(x_t) \right\rangle \nonumber\\
&&+ \frac{L\eta_t^2}{2}\lVert  \nabla f(x_t)\rVert^2 \nonumber\\
&&+  \frac{L\eta_t^2}{2}{\mathbf E}_t\left[\left\lVert\frac{1}{M}\sum_{i=1}^Mg_{t,i} - \nabla f(x_t)\right\rVert^2\right].\label{bound}
\end{eqnarray}
Following \cite{Zhengetal}, we assume that $\{g_{t,i} - \nabla f (x_t)\}_{1 \le i \le M}$ are independent random vectors. Then the assumption ${\mathbf E}_t[g_{t,i}] = \nabla f (x_t) $ of Assumption \ref{stobounded} implies that $g_{t,i} - \nabla f (x_t)$ are random vectors with $0$ means. Therefore 
\begin{eqnarray*}
{\mathbf E}_t\left[\left \lVert \sum_{i=1}^M (g_{t,i} \!-\! \nabla f(x_t))\right \rVert^2\right] \!= \! \sum_{i=1}^M{\mathbf E}_t\!\left[\left \lVert g_{t,i} \!-\! \nabla f(x_t)\right \rVert^2\right]
\end{eqnarray*} 
and hence we have 
\begin{eqnarray*}
\lefteqn{{\mathbf E}_t\left[\left\lVert\frac{1}{M}\sum_{i=1}^Mg_{t,i} - \nabla f(x_t)\right\rVert^2\right] }\\
&=& {\mathbf E}_t\left[\left\lVert\frac{1}{M}\sum_{i=1}^M\bigg(g_{t,i} - \nabla f(x_t)\bigg)\right\rVert^2\right] \\
& = & \frac{1}{M^2}{\mathbf E}_t\left[\left\lVert\sum_{i=1}^M\bigg(g_{t,i} - \nabla f(x_t)\bigg)\right\rVert^2\right] \\
& = & \frac{1}{M^2}\sum_{i=1}^M {\mathbf E}_t\left[\left\lVert g_{t,i} - \nabla f(x_t)\right\rVert^2\right] \\
&\le& \frac{\sigma^2M}{M^2} = \frac{\sigma^2}{M}.
\end{eqnarray*}
Substituting the above bound to (\ref{bound}) gives us
\begin{align}\label{sp1}
{\mathbf E}_t[f(\tilde x_{t+1})] 
& \le& f(\tilde x_t) - \eta_t\left\langle \nabla f(\tilde x_{t}), \nabla f(x_t) \right\rangle \nonumber\\
&&+ \frac{L\eta_t^2}{2}\lVert  \nabla f(x_t)\rVert^2 +  \frac{L\eta_t^2\sigma^2}{2M}.
\end{align}
Moreover, we have 
\begin{flalign}
\lefteqn{-\eta_t\left\langle \nabla f(\tilde x_{t}), \nabla f(x_t)\right\rangle}\nonumber\\ 
& =  \eta_t\!\left\langle  \nabla f(x_t) \!-\! \nabla f(\tilde x_{t}), \nabla f(x_t) \right\rangle \!-\! \eta_t\!\left\langle \nabla f(x_t), \nabla f(x_t) \right\rangle \nonumber\\
& =   \eta_t\left\langle  \nabla f(x_t) - \nabla f(\tilde x_{t}), \nabla f(x_t) \right\rangle -\eta_t\lVert  \nabla f(x_t)\rVert^2 \nonumber\\
&{\le}  \frac{\eta_t\rho}{2}\lVert \nabla f(x_t)\rVert^2 \! +\! \frac{\eta_t}{2\rho}\lVert  \nabla f(x_t)\! - \!\nabla f(\tilde x_{t})\rVert^2\nonumber\\
& \quad- \eta_t\lVert  \nabla f(x_t)\rVert^2 \label{pt9}\\
&{\le}  -\eta_t \left( 1-\frac{\rho}{2}\right)\lVert \nabla f(x_t)\rVert^2 +  \frac{\eta_tL^2}{2\rho}\lVert  x_t - \tilde x_{t}\rVert^2,\label{pt10}
\end{flalign}
where $(\ref{pt9})$ is by the fact that $\langle a,b \rangle \le (\rho/2) \lVert a \rVert^2 + (\rho^{-1}/2) \lVert b \rVert^2$  for all $a,b$ and real number $\rho>0$, and $(\ref{pt10})$ is by Assumption \ref{smoothass}.
Replacing (\ref{pt10}) to (\ref{sp1}) gives us
\begin{flalign*}
{\mathbf E}_t[f(\tilde x_{t+1})] 
& \le f(\tilde x_t)  -\eta_t \left( 1 - \frac{L\eta_t + \rho}{2}\right)\lVert \nabla f(x_t)\rVert^2 \\
&+  \frac{\eta_tL^2}{2\rho}\lVert  x_t - \tilde x_{t}\rVert^2 +  \frac{L\eta_t^2\sigma^2}{2M}.
\end{flalign*}
Taking $\rho = \frac{1}{2}$, we have
\begin{flalign}\label{rho1/2} 
{\mathbf E}_t[f(\tilde x_{t+1})] 
& \le f(\tilde x_t)  -\eta_t \left( \frac{3}{4} - \frac{L\eta_t }{2}\right)\lVert \nabla f(x_t)\rVert^2 \nonumber\\
&\quad+ \eta_tL^2\lVert  x_t - \tilde x_{t}\rVert^2 +  \frac{L\eta_t^2\sigma^2}{2M}.
\end{flalign}
Since $x_t - \tilde x_t =  \eta_{t-1}\left( \tilde e_t + \frac{1}{M}\sum_{i=1}^M e_{t,i}\right)$ by (\ref{lemofZheng}), after rearranging the terms and taking total expectation, we obtain
\begin{eqnarray*}
\lefteqn{\eta_t\left( \frac{3}{4} -\frac{L\eta_t }{2}\right){\mathbf E}[\lVert  \nabla f(x_t)\rVert^2] }\\
& \le&  {\mathbf E}[f(\tilde x_t) - f(\tilde x_{t+1})]  + \frac{L\eta_t^2\sigma^2}{2M} \\
&&+\eta_t\eta_{t-1}^2L^2{\mathbf E}\left[\left\lVert \tilde e_t + \frac{1}{M}\sum_{i=1}^{M}e_{t,i}\right\rVert^2\right]. 
\end{eqnarray*}
Applying Theorem \ref{error} gives us
\begin{flalign}\label{later1}
\lefteqn{\eta_t\left( \frac{3}{4} -\frac{L\eta_t }{2}\right){\mathbf E}[\lVert  \nabla f(x_t)\rVert^2] }\\
&\le {\mathbf E}[f(\tilde x_t) - f(\tilde x_{t+1})] + \frac{L\eta_t^2\sigma^2}{2M}\nonumber\\
&+ \eta_t\eta_{t-1}^2\frac{2(2-\delta)(1-\delta)G^2L^2}{\delta}\sum_{k=0}^{t-1}\frac{\eta_{t-1-k}^2}{\eta_{t-1}^2}\alpha^k\nonumber\\
&+\eta_t\eta_{t\!-\!1}^2\frac{4(1\!-\!\delta)(2\!-\!\delta)^3G^2L^2}{\delta^2}\sum_{j\!=\!0}^{t\!-\!1} \alpha^{t\!-\!1\!-\!j} \sum_{k=0}^{j}\frac{\eta_{j\!-\!k}^2}{\eta_{t\!-\!1}^2}\alpha^k,\nonumber
\end{flalign}
where $\alpha = 1-\frac{\delta}{2}$.
Since $\eta_t <\frac{3}{2L}, \forall t$, we have  $$\sum_{k=0}^{T-1}\frac{\eta_k}{4}(3-2L\eta_k) >0.$$ 
Taking summation and dividing by $\sum_{k=0}^{T-1}\frac{\eta_k}{4}(3-2L\eta_k)$, (\ref{later1}) yields
\begin{flalign*}
\lefteqn{\sum_{t=0}^{T-1}\frac{\eta_t\left( 3-2L\eta_t\right){\mathbf E}[\lVert  \nabla f(x_t)\rVert^2]}{\sum_{k=0}^{T-1}\eta_k(3-2L\eta_k)}} \\
&\le \frac{4(f(x_0) - f^{\star})}{\sum_{k=0}^{T-1}\eta_k(3-2L\eta_k)} \\
&\quad+\frac{2L\sigma^2}{M}\sum_{t=0}^{T-1}\frac{\eta_t^2}{\sum_{k=0}^{T-1}\eta_k(3-2L\eta_k)}\\
&\quad+\frac{8(2-\delta)(1-\delta)G^2L^2}{\delta\sum_{k=0}^{T-1}\eta_k(3-2L\eta_k)}\sum_{t=0}^{T-1}\eta_t\eta_{t-1}^2\sum_{k=0}^{t-1}\frac{\eta_{t-1-k}^2}{\eta_{t-1}^2}\alpha^k\\
&\quad+\frac{16(1-\delta)(2-\delta)^3G^2L^2}{\delta^2 \sum_{k=0}^{T-1}\eta_k(3-2L\eta_k)} \times\\
&\quad\quad\quad\quad\quad\sum_{t=0}^{T-1}\eta_t\eta_{t-1}^2\sum_{j=0}^{t-1} \alpha^{t-1-j} \sum_{k=0}^{j}\frac{\eta_{j-k}^2}{\eta_{t-1}^2}\alpha^k.
\end{flalign*}
Following Zheng et al. \cite{Zhengetal}, let $o\in \{0,...,T-1\}$ be an index such that 
$$\Pr(o=k) = \frac{\eta_k(3-2L\eta_k)}{\sum_{t=0}^{T-1}\eta_t(3-2L\eta_t)}.$$
Then 
\begin{flalign*}
\lefteqn{{\mathbf E}[\lVert \nabla f(x_o)\rVert^2] }\\
&= \frac{1}{\sum_{k\!=\!0}^{T\!-\!1}\eta_k(3\!-\!2L\eta_k)} \sum_{t\!=\!0}^{T\!-\!1}\eta_t(3\!-\!2L\eta_t){\mathbf E}[\lVert \nabla f(x_t)\rVert^2]
\end{flalign*}
and we obtain the result stated in Theorem \ref{mainthmcase1}
\end{proof}
\medskip \noindent
The following corollary establishes the convergence rate $O(\frac{1}{\sqrt{MT}})$ of Algorithm \ref{dist-EF-SGD} when the learning rate is decreasing.

\begin{cor}(Convergence rate with decreasing learning rate) \label{decreasing} Under the assumptions of Theorem \ref{mainthmcase1}, if $\{\eta_t \}$ is a decreasing sequence such that 
$$\eta_t =\frac{1}{\frac{((t+1)T)^{1/4}}{\sqrt M}+ T^{1/3}}$$ 
with sufficiently large $T$. Then 
\begin{eqnarray*}
\lefteqn{{\mathbf E}[\lVert \nabla f(x_o)\rVert^2]}\\
&\le&2\bigg(\frac{1}{\sqrt{MT}} + \frac{1}{T^{2/3}} \bigg)\bigg[f(x_0) - f^{\star} +L\sigma^2\\
&&\quad\quad\quad+\frac{4(1-\delta)(2-\delta)G^2L^2}{\delta^2} \bigg(1+\frac{4}{\delta^2} \bigg)\bigg],
\end{eqnarray*}
which yields ${\mathbf E}[\lVert \nabla f(x_o)\rVert^2] \le O(\frac{1}{\sqrt{MT}})$.
\end{cor}

\begin{proof} ~Following \cite{Zhengetal}, assume that $T \ge 16L^4M^2 $, we have 
\begin{eqnarray*}
\eta_t   =  \frac{1}{\frac{((t+1)T)^{1/4}}{\sqrt M}+ T^{1/3}} &\le& \frac{\sqrt M}{((t+1)T)^{1/4}}\\
& \le & \frac{\sqrt M}{(t+1)^{1/4}(16L^4M^2)^{1/4}}\\
& = & \frac{1}{(t+1)^{1/4}2L} \le \frac{1}{2L}.
\end{eqnarray*}
Therefore $\eta_t <\frac{3}{2L} \forall t\ge 0$, which satisfies the assumption of Theorem \ref{mainthmcase1} on $\{\eta_t\}$. 
Recall that by Theorem \ref{mainthmcase1}, we have
\begin{flalign}
\lefteqn{{\mathbf E}[\lVert \nabla f(x_o)\rVert^2]}\nonumber\\
&\le \frac{4(f(x_0) \!-\! f^{\star})}{\sum_{k\!=\!0}^{T\!-\!1}\eta_k(3\!-\!2L\eta_k)} \!+\!\frac{2L\sigma^2}{M{\sum_{k=0}^{T-1}\eta_k(3\!-\!2L\eta_k)}}\sum_{t\!=\!0}^{T\!-\!1}\eta_t^2\nonumber\\
&\quad+\frac{8(1-\delta)(2-\delta)G^2L^2}{\delta\sum_{k=0}^{T-1}\eta_k(3-2L\eta_k)}\sum_{t=0}^{T-1}\eta_t\eta_{t-1}^2\sum_{k=0}^{t-1}\frac{\eta_{t-1-k}^2}{\eta_{t-1}^2}\alpha^k\nonumber\\
&\quad+\frac{16(1-\delta)(2-\delta)^3G^2L^2}{\delta^2\sum_{k=0}^{T-1}\eta_k(3-2L\eta_k)}\times\nonumber\\
&\quad\quad\quad\quad\sum_{t=0}^{T-1}\eta_t\eta_{t-1}^2\sum_{j=0}^{t-1} \alpha^{t-1-j} \sum_{k=0}^{j}\frac{\eta_{j-k}^2}{\eta_{t-1}^2}\alpha^k, \label{forsub0}
\end{flalign}
where $\alpha = 1-\frac{\delta}{2}$ and $o\in \{0,...,T-1\}$ is an index such that 
$$\Pr(o=k) = \frac{\eta_k(3-2L\eta_k)}{\sum_{t=0}^{T-1}\eta_t(3-2L\eta_t)}, \forall k = 0,...,T-1.$$ 
Since 
$$3-2L\eta_t  \ge 3-\frac{1}{(t+1)^{1/4}} \ge 2,$$
we have
\begin{eqnarray}\label{forsub1}
\frac{1}{\sum_{t=0}^{T-1}\eta_t(3-2L\eta_t)} \le \frac{1}{2\sum_{t=0}^{T-1}\eta_t}.
\end{eqnarray}
Moreover, we have 
\begin{eqnarray}\label{forsub2}
\eta_t\eta_{t-1}^2\sum_{k=0}^{t-1}\frac{\eta_{t-1-k}^2}{\eta_{t-1}^2}\alpha^k = \sum_{k=0}^{t-1}\eta_t\eta_{t-1-k}^2\alpha^k
\end{eqnarray}
and 
\begin{eqnarray}
\eta_t\eta_{t-1}^2\sum_{j=0}^{t-1} \alpha^{t-1-j}\sum_{k=0}^{j}\frac{\eta_{j-k}^2}{\eta_{t-1}^2}\alpha^k \nonumber\\
= \sum_{j=0}^{t-1} \alpha^{t-1-j} \sum_{k=0}^{j} \eta_t\eta_{j-k}^2\alpha^k.\label{forsub3}
\end{eqnarray}
Substituting (\ref{forsub1}), (\ref{forsub2}), and (\ref{forsub3}) to (\ref{forsub0}) gives us
\begin{eqnarray}
\lefteqn{{\mathbf E}[\lVert \nabla f(x_o)\rVert^2]}\nonumber\\
&\le& \frac{2(f(x_0) - f^{\star})}{\sum_{t=0}^{T-1}\eta_t} +\frac{L\sigma^2}{M\sum_{t=0}^{T-1}\eta_t}\sum_{t=0}^{T-1}\eta_t^2\nonumber\\
&&+\frac{4(1-\delta)(2-\delta)G^2L^2}{\delta \sum_{t=0}^{T-1}\eta_t}\sum_{t=0}^{T-1}\sum_{k=0}^{t-1}\eta_t\eta_{t-1-k}^2\alpha^k\nonumber\\
&&+\frac{8(1-\delta)(2-\delta)^3G^2L^2}{\delta^2 \sum_{t=0}^{T-1}\eta_t}\times\nonumber\\
&&\quad\quad\quad\sum_{t=0}^{T-1}\sum_{j=0}^{t-1} \alpha^{t-1-j} \sum_{k=0}^{j} \eta_t\eta_{j-k}^2\alpha^k,\label{forsub4}
\end{eqnarray}
%Next, we will bound the third and the fourth terms of (\ref{forsub4}). 
Because 
 \begin{eqnarray}
 \eta_t\eta_{t-1-k}^2 &\le& \eta_{t-1-k}^3\nonumber\\
 & = &  \frac{1}{\left(\frac{((t-k)T)^{1/4}}{\sqrt M}+ T^{1/3}\right)^3}\nonumber\\
 & \le &  \frac{1}{{(T^{1/3})^3}} = \frac{1}{T} \label{1overT}
 \end{eqnarray}
 and 
\begin{eqnarray}\label{last}
\sum_{k=0}^{t-1} \alpha^k \le \sum_{k\ge0} \alpha^k = \frac{1}{1-\alpha} = \frac{2}{\delta},
\end{eqnarray}
we obtain
 \begin{eqnarray*}
 \sum_{t=0}^{T-1}\sum_{k=0}^{t-1}\eta_t\eta_{t-1-k}^2 \alpha^k
 &\le & \frac{1}{T}\sum_{t=0}^{T-1}\sum_{k=0}^{t-1} \alpha^k\\
 & \le & \frac{1}{T}\sum_{t=0}^{T-1} \frac{2}{\delta} =  \frac{2}{\delta}.
 \end{eqnarray*}
By the same reason as in (\ref{1overT}) and (\ref{last}), we have $\eta_t \eta_{j-k}^2 \le \frac{1}{T}$, $\sum_{j=0}^{t-1} \alpha^{t-1-j} \le\frac{2}{\delta}$, and $ \sum_{k=0}^{j}\alpha^k \le\frac{2}{\delta}$. Therefore
 \begin{eqnarray*}
 \lefteqn{\sum_{t=0}^{T-1}\sum_{j=0}^{t-1} \alpha^{t-1-j} \sum_{k=0}^{j}\eta_t \eta_{j-k}^2\alpha^k}\\
 &\le &  \frac{1}{T} \sum_{t=0}^{T-1}\sum_{j=0}^{t-1} \alpha^{t-1-j} \sum_{k=0}^{j}\alpha^k\\
 &\le &  \frac{1}{T} \sum_{t=0}^{T-1}\sum_{j=0}^{t-1} \alpha^{t-1-j}\left(\frac{2}{\delta}\right)\\
 &\le &  \frac{1}{T} \sum_{t=0}^{T-1}\frac{4}{\delta^2} = \frac{4}{\delta^2}.
 \end{eqnarray*} 
 Moreover, we have 
 \begin{eqnarray*}
 \sum_{t=0}^{T-1}\eta_t^2 &=& \sum_{t=0}^{T-1} \frac{1}{\left(\frac{((t+1)T)^{1/4}}{\sqrt M}+ T^{1/3}\right)^2} \\
 &\le& \sum_{t=0}^{T-1}\frac{1}{\left(\frac{((t+1)T)^{1/4}}{\sqrt M}\right)^2} \\
 &=& \sum_{t=0}^{T-1}\frac{M}{[(t+1)T]^{1/2}} \\
 &=& \frac{M}{\sqrt T}\sum_{t=1}^{T}\frac{1}{\sqrt t} \\
 & \le &2M,
 \end{eqnarray*}
 where the last inequality is by the fact that $\sum_{t=1}^{T}\frac{1}{\sqrt t} \le 2\sqrt T$. 
 Therefore 
 \begin{eqnarray*}
{\mathbf E}[\lVert \nabla f(x_o)\rVert^2]
&\le& \frac{2(f(x_0) - f^{\star})}{\sum_{t=0}^{T-1}\eta_t} +\frac{2L\sigma^2}{\sum_{t=0}^{T-1}\eta_t}\\
&&+\frac{8(1-\delta)(2-\delta)G^2L^2}{\delta^2 \sum_{t=0}^{T-1}\eta_t}\\
&&+\frac{32(1-\delta)(2-\delta)^3G^2L^2}{\delta^4 \sum_{t=0}^{T-1}\eta_t}.
\end{eqnarray*}
Furthermore, becauce
 \begin{eqnarray*}
 \sum_{t=0}^{T-1}\eta_t = \sum_{t=0}^{T-1}\frac{1}{\frac{((t+1)T)^{1/4}}{(\sqrt M)} +  T^{1/3}}
 & \ge&  \sum_{t=0}^{T-1}\frac{1}{\frac{\sqrt T}{\sqrt M} +  T^{1/3}}\\
 & = & \frac{1}{\frac{1}{\sqrt {MT}} + T^{-2/3}},
 \end{eqnarray*}
 we obtain 
 \begin{eqnarray*} 
 \frac{1}{\sum_{t=0}^{T-1}\eta_t}
 \le \frac{1}{\sqrt{MT}} + \frac{1}{T^{2/3}}.
 \end{eqnarray*} 
Therefore
\begin{eqnarray*}
\lefteqn{{\mathbf E}[\lVert \nabla f(x_o)\rVert^2]}\\
&\le&2\bigg(\frac{1}{\sqrt{MT}} + \frac{1}{T^{2/3}} \bigg)\bigg[f(x_0) - f^{\star} +L\sigma^2\\
&&\quad\quad+\frac{4(1-\delta)(2-\delta)G^2L^2}{\delta^2} \bigg(1+\frac{4}{\delta^2} \bigg)\bigg]
\end{eqnarray*}
and hence Corollary \ref{decreasing} follows.
\end{proof}

%%%%%%%%%%%%%%%%%%%%%%%%
\section{Conclusion}
We show that the convergence proof of dist-EF-SGD of  Zheng et al.\cite{Zhengetal} is problematic when the sequence of learning rate is decreasing. We explicitly provide counter-examples with certain decreasing sequences of learning rate to show the issue  in the proof of Zheng et al. \cite{Zhengetal}. We fix the issue by providing a new error bound and a new convergence theorem for the dist-EF-SGD algorithm, which helps recover its mathematical foundation.

%----------------------------- END BODY OF TEXT -----------------------------
\section*{Acknowledgements}
We are grateful to Shuai Zheng for his communication and verification. %We also thank the anonymous reviewers for their careful comments. The work of Le Trieu Phong was supported in part by JST CREST under Grant JPMJCR19F6.

%\section*{References}

%\bibliography{ref_papers}{}

\begin{thebibliography}{10}

\bibitem{Zhengetal}
Shuai Zheng, Ziyue Huang, and James~T. Kwok.
\newblock Communication-efficient distributed blockwise momentum {SGD} with
  error-feedback.
\newblock In {\em Advances in Neural Information Processing Systems 32: Annual
  Conference on Neural Information Processing Systems 2019, NeurIPS 2019},
  pages 11446--11456, 2019.
\newblock [Online]. Available: https://arxiv.org/abs/1905.10936.

\bibitem{BernsteinZAA19}
Jeremy Bernstein, Jiawei Zhao, Kamyar Azizzadenesheli, and Anima Anandkumar.
\newblock {signSGD} with majority vote is communication efficient and fault
  tolerant.
\newblock In {\em 7th International Conference on Learning Representations,
  {ICLR} 2019}, 2019.

\bibitem{DebrajBasu_nips2019}
Debraj Basu, Deepesh Data, Can Karakus, and Suhas~N. Diggavi.
\newblock Qsparse-local-sgd: Distributed {SGD} with quantization,
  sparsification and local computations.
\newblock In {\em Advances in Neural Information Processing Systems 32: Annual
  Conference on Neural Information Processing Systems 2019, NeurIPS 2019},
  pages 14668--14679, 2019.

\bibitem{NIPS2019_9571}
Thijs Vogels, Sai~Praneeth Karimireddy, and Martin Jaggi.
\newblock {PowerSGD}: Practical low-rank gradient compression for distributed
  optimization.
\newblock In {\em Advances in Neural Information Processing Systems 32}, pages
  14236--14245. Curran Associates, Inc., 2019.

\bibitem{StichCJ18}
Sebastian~U. Stich, Jean{-}Baptiste Cordonnier, and Martin Jaggi.
\newblock Sparsified {SGD} with memory.
\newblock In {\em Advances in Neural Information Processing Systems 31: Annual
  Conference on Neural Information Processing Systems 2018}, pages 4452--4463,
  2018.

\bibitem{tang19d_icml}
Hanlin Tang, Chen Yu, Xiangru Lian, Tong Zhang, and Ji~Liu.
\newblock $\texttt{DoubleSqueeze}$: Parallel stochastic gradient descent with
  double-pass error-compensated compression.
\newblock In Kamalika Chaudhuri and Ruslan Salakhutdinov, editors, {\em
  Proceedings of the 36th International Conference on Machine Learning},
  volume~97 of {\em Proceedings of Machine Learning Research}, pages
  6155--6165. PMLR, 2019.

\bibitem{pmlr-v108-liu20a}
Xiaorui Liu, Yao Li, Jiliang Tang, and Ming Yan.
\newblock A double residual compression algorithm for efficient distributed
  learning.
\newblock volume 108 of {\em Proceedings of Machine Learning Research}, pages
  133--143, Online, 26--28 Aug 2020. PMLR.

\bibitem{9051706}
T.~T. {Phuong} and L.~T. {Phong}.
\newblock Distributed sgd with flexible gradient compression.
\newblock {\em IEEE Access}, 8:64707--64717, 2020.

\bibitem{EFSGD}
Sai~Praneeth Karimireddy, Quentin Rebjock, Sebastian~U. Stich, and Martin
  Jaggi.
\newblock Error feedback fixes sign{SGD} and other gradient compression
  schemes.
\newblock In {\em Proceedings of the 36th International Conference on Machine
  Learning, {ICML} 2019}, pages 3252--3261, 2019.
\newblock [Online]. Available: https://arxiv.org/abs/1901.09847.

\bibitem{Nesterov}
Yurii Nesterov.
\newblock Introductory lectures on convex optimization.
\newblock Volume 87, Springer Science and Business Media, Springer US, Boston,
  MA, 2004.

\end{thebibliography}
%\bibliographystyle{unsrt} %  abbrv
\end{document}